\documentclass{amsart}
\usepackage{amssymb}
\usepackage{amsfonts}

\newtheorem{theorem}{Theorem}
\theoremstyle{plain}

\newtheorem{corollary}{Corollary}

\newtheorem{definition}{Definition}

\newtheorem{lemma}{Lemma}

\newtheorem{proposition}{Proposition}
\newtheorem{remark}{Remark}

\numberwithin{equation}{section}
\input{tcilatex}

\begin{document}
\title[Some Companions of Ostrowski type inequality]{Some Companions of
Ostrowski type inequality for $s-$convex and $s-$concave functions with
applications.}
\author{M.Emin \"{O}zdemir$^{\blacklozenge }$}
\address{$^{\blacklozenge }$Atat\"{u}rk University, K.K. Education Faculty,
Department of Mathematics, Erzurum 25240, Turkey}
\email{emos@atauni.edu.tr}
\author{Merve Avci Ardic$^{\bigstar \diamondsuit }$}
\address{$^{\bigstar }$Adiyaman University, Faculty of Science and Arts,
Department of Mathematics, Adiyaman 02040, Turkey}
\email{mavci@posta.adiyaman.edu.tr}
\thanks{$^{\lozenge }$Corresponding Author}

\begin{abstract}
In this paper, we obtain some companions of Ostrowski type inequality for
absolutely continuous functions whose second derivatives absolute value are $%
s-$convex and $s-$concave. Finally, we gave some applications for special
means.
\end{abstract}

\maketitle

\section{\protect\bigskip introduction}

In \cite{HM}, Hudzik and Maligranda considered among others the class of
functions which are $s-$convex in the second sense. This class is defined in
the following way:

\begin{definition}
\label{def 1.1} A function $f:%
\mathbb{R}
^{+}\rightarrow 
\mathbb{R}
,$ where $%
\mathbb{R}
^{+}=[0,\infty ),$ is said to be $s-$convex in the second sense if%
\begin{equation*}
f(\alpha x+\beta y)\leq \alpha ^{s}f(x)+\beta ^{s}f(y)
\end{equation*}%
for all $x,y\in \lbrack 0,\infty ),$ $\alpha ,\beta \geq 0$ with $\alpha
+\beta =1$ and for some fixed $s\in (0,1].$
\end{definition}

The class of $s-$convex functions in the second sense is usually denoted by $%
K_{s}^{2}.$ It can be easily seen that for $s=1,$ $s-$convexity reduces to
ordinary convexity of functions defined on $[0,\infty ).$

In \cite{DF}, Dragomir and Fitzpatrick proved a variant of Hadamard's
inequality which holds for $s-$convex functions in the second sense:

\begin{theorem}
\label{teo 1.1} Suppose that $f:[0,\infty )\rightarrow \lbrack 0,\infty )$
is an $s-$convex function in the second sense, where $s\in \left( 0,1\right) 
$ and let $a,b\in \lbrack 0,\infty ),$ $a<b.$ If $f\in L^{1}[a,b],$ then the
following inequalities hold:%
\begin{equation}
2^{s-1}f\left( \frac{a+b}{2}\right) \leq \frac{1}{b-a}\int_{a}^{b}f(x)dx\leq 
\frac{f(a)+f(b)}{s+1}.  \label{1.1}
\end{equation}%
The constant $k=\frac{1}{s+1}$ is the best possible in the second inequality
in (\ref{1.1}). The above inequalities are sharp.
\end{theorem}

The following inequality is well known as Ostrowski's inequality in the
literature \cite{O}:

\begin{theorem}
\label{teo 2.0} Let $f:I\subset 
\mathbb{R}
\rightarrow 
\mathbb{R}
$ be a differentiable mapping on $I^{\circ },$ the interior of the interval $%
I,$ such that $f^{\prime }\in L[a,b],$ where $a,b\in I$ with $a<b.$ If $%
\left\vert f^{\prime }(x)\right\vert \leq M,$ then the following inequality,%
\begin{equation*}
\left\vert f(x)-\frac{1}{b-a}\int_{a}^{b}f(x)dx\right\vert \leq M\left(
b-a\right) \left[ \frac{1}{4}+\frac{\left( x-\frac{a+b}{2}\right) ^{2}}{%
\left( b-a\right) ^{2}}\right]
\end{equation*}%
holds for all $x\in \lbrack a,b].$ The constant $\frac{1}{4}$ is the best
possible in the sense that it can not be replaced by a smaller constant.
\end{theorem}

In \cite{SSO}, Set et al. proved some inequalities for $s-$concave and
concave functions via following Lemma:

\begin{lemma}
\label{lem 1.0} Let $f:I\subseteq 
\mathbb{R}
\rightarrow 
\mathbb{R}
$ be a twice differentiable function on $I^{\circ }$ with $f^{\prime \prime
}\in L_{1}[a,b],$ then 
\begin{eqnarray*}
&&\frac{1}{b-a}\int_{a}^{b}f(u)du-f(x)+\left( x-\frac{a+b}{2}\right)
f^{\prime }(x) \\
&=&\frac{\left( x-a\right) ^{3}}{2\left( b-a\right) }\int_{0}^{1}t^{2}f^{%
\prime \prime }\left( tx+\left( 1-t\right) a\right) dt+\frac{\left(
b-x\right) ^{3}}{2\left( b-a\right) }\int_{0}^{1}t^{2}f^{\prime \prime
}\left( tx+\left( 1-t\right) b\right) dt.
\end{eqnarray*}
\end{lemma}

\begin{theorem}
\label{teo 1.2} Let $f:I\subseteq \lbrack 0,\infty )\rightarrow 
\mathbb{R}
$ be a twice differentiable function on $I^{\circ }$ such that $f^{\prime
\prime }\in L_{1}[a,b]$ where $a,b\in I$ with $a<b.$ If $\left\vert
f^{\prime \prime }\right\vert ^{q}$ is $s-$concave in the second sense on $%
[a,b]$ for some fixed $s\in (0,1],$ $p,q>1$ and $\frac{1}{p}+\frac{1}{q}=1,$
then the following inequality holds:%
\begin{eqnarray}
&&\left\vert \frac{1}{b-a}\int_{a}^{b}f(u)du-f(x)+\left( x-\frac{a+b}{2}%
\right) f^{\prime }(x)\right\vert  \label{1.2} \\
&\leq &\frac{2^{\frac{\left( s-1\right) }{q}}}{\left( 2p+1\right) ^{\frac{1}{%
p}}\left( b-a\right) }\left( \frac{\left( x-a\right) ^{3}\left\vert
f^{\prime \prime }\left( \frac{x+a}{2}\right) \right\vert +\left( b-x\right)
^{3}\left\vert f^{\prime \prime }\left( \frac{x+a}{2}\right) \right\vert }{2}%
\right)  \notag
\end{eqnarray}%
for each $x\in \left[ a,b\right] .$
\end{theorem}

\begin{corollary}
\label{co 1.1} If in (\ref{1.2}), we choose $x=\frac{a+b}{2},$ then we have 
\begin{equation*}
\left\vert \frac{1}{b-a}\int_{a}^{b}f(u)du-f\left( \frac{a+b}{2}\right)
\right\vert \leq \frac{2^{\frac{\left( s-1\right) }{q}}\left( b-a\right) ^{2}%
}{16\left( 2p+1\right) ^{\frac{1}{p}}}\left[ \left\vert f^{\prime \prime
}\left( \frac{3a+b}{4}\right) \right\vert +\left\vert f^{\prime \prime
}\left( \frac{a+3b}{4}\right) \right\vert \right] .
\end{equation*}%
For instance, if $s=1,$ then we have 
\begin{equation*}
\left\vert \frac{1}{b-a}\int_{a}^{b}f(u)du-f\left( \frac{a+b}{2}\right)
\right\vert \leq \frac{\left( b-a\right) ^{2}}{16\left( 2p+1\right) ^{\frac{1%
}{p}}}\left[ \left\vert f^{\prime \prime }\left( \frac{3a+b}{4}\right)
\right\vert +\left\vert f^{\prime \prime }\left( \frac{a+3b}{4}\right)
\right\vert \right] .
\end{equation*}
\end{corollary}

In \cite{Z}, Liu introduced some companions of an Ostrowski type inequality
for functions whose first derivative are absolutely continuous.

In this paper, we established some companions of Ostrowski type inequality
for absolutely continuous functions whose second derivatives absolute value
are $s-$convex and $s-$concave which are reduce the results proved in \cite%
{MM}, for $s=1.$

In order to prove our main results we need the following Lemma \cite{Z}:

\begin{lemma}
\label{lem 1.1} Let $f:[a,b]\rightarrow 
\mathbb{R}
$ be such that the derivative $f^{\prime }$ is absolutely continuous on $%
[a,b].$ Then we have the inequality%
\begin{eqnarray*}
&&\frac{1}{b-a}\int_{a}^{b}f(t)dt-\frac{1}{2}\left[ f(x)+f(a+b-x)\right] \\
&&+\frac{1}{2}\left( x-\frac{3a+b}{4}\right) \left[ f^{\prime }(x)-f^{\prime
}(a+b-x)\right] \\
&=&\frac{1}{2\left( b-a\right) }\left[ \int_{a}^{x}\left( t-a\right)
^{2}f^{\prime \prime }(t)dt+\int_{x}^{a+b-x}\left( t-\frac{a+b}{2}\right)
^{2}f^{\prime \prime }(t)dt\right. \\
&&\left. +\int_{a+b-x}^{b}\left( t-b\right) ^{2}f^{\prime \prime }(t)dt 
\right]
\end{eqnarray*}%
for all $x\in \left[ a,\frac{a+b}{2}\right] .$
\end{lemma}

\section{main results}

\begin{theorem}
\label{teo 2.1} Let $f:[a,b]\rightarrow 
\mathbb{R}
$ be a function such that $f^{\prime }$ is absolutely continuous on $[a,b],$ 
$f^{\prime \prime }\in L_{1}[a,b].$ If $\left\vert f^{\prime \prime
}\right\vert $ is $s-$convex on $[a,b],$ then we have the following
inequality:%
\begin{eqnarray*}
&&\left\vert \frac{1}{b-a}\int_{a}^{b}f(t)dt-\frac{1}{2}\left[ f(x)+f(a+b-x)%
\right] \right. \\
&&\left. +\frac{1}{2}\left( x-\frac{3a+b}{4}\right) \left[ f^{\prime
}(x)-f^{\prime }(a+b-x)\right] \right\vert \\
&\leq &\frac{\left( x-a\right) ^{3}}{\left( s+1\right) \left( s+2\right)
\left( s+3\right) (b-a)}\left[ \left\vert f^{\prime \prime }(a)\right\vert
+\left\vert f^{\prime \prime }(b)\right\vert \right] \\
&&+\frac{4\left( s^{2}+3s+2\right) \left( x-a\right) ^{3}+\left(
s^{2}+s+2\right) \left( a+b-2x\right) ^{3}}{8\left( s+1\right) \left(
s+2\right) \left( s+3\right) (b-a)}\left[ \left\vert f^{\prime \prime
}(x)\right\vert +\left\vert f^{\prime \prime }(a+b-x)\right\vert \right]
\end{eqnarray*}%
for all $x\in \left[ a,\frac{a+b}{2}\right] .$
\end{theorem}

\begin{proof}
Using Lemma \ref{lem 1.1} and the property of the modulus we have%
\begin{eqnarray*}
&&\left\vert \frac{1}{b-a}\int_{a}^{b}f(t)dt-\frac{1}{2}\left[ f(x)+f(a+b-x)%
\right] \right. \\
&&\left. +\frac{1}{2}\left( x-\frac{3a+b}{4}\right) \left[ f^{\prime
}(x)-f^{\prime }(a+b-x)\right] \right\vert \\
&\leq &\frac{1}{2(b-a)}\left[ \int_{a}^{x}\left( t-a\right) ^{2}\left\vert
f^{\prime \prime }(t)\right\vert dt+\int_{x}^{a+b-x}\left( t-\frac{a+b}{2}%
\right) ^{2}\left\vert f^{\prime \prime }(t)\right\vert dt\right. \\
&&\left. +\int_{a+b-x}^{b}\left( t-b\right) ^{2}\left\vert f^{\prime \prime
}(t)\right\vert dt\right] .
\end{eqnarray*}%
Since $\left\vert f^{\prime \prime }\right\vert $ is $s-$convex on $[a,b],$
we have%
\begin{equation*}
\left\vert f^{\prime \prime }(t)\right\vert \leq \left( \frac{t-a}{x-a}%
\right) ^{s}\left\vert f^{\prime \prime }(x)\right\vert +\left( \frac{x-t}{%
x-a}\right) ^{s}\left\vert f^{\prime \prime }(a)\right\vert ,\text{ \ \ \ \ }%
t\in \lbrack a,x];
\end{equation*}

\begin{equation*}
\left\vert f^{\prime \prime }(t)\right\vert \leq \left( \frac{t-x}{a+b-2x}%
\right) ^{s}\left\vert f^{\prime \prime }(a+b-x)\right\vert +\left( \frac{%
a+b-x-t}{a+b-2x}\right) ^{s}\left\vert f^{\prime \prime }(x)\right\vert ,%
\text{ \ \ \ \ }t\in (x,a+b-x]
\end{equation*}%
and%
\begin{equation*}
\left\vert f^{\prime \prime }(t)\right\vert \leq \left( \frac{t-a-b+x}{x-a}%
\right) ^{s}\left\vert f^{\prime \prime }(b)\right\vert +\left( \frac{b-t}{%
x-a}\right) ^{s}\left\vert f^{\prime \prime }(a+b-x)\right\vert ,\text{ \ \
\ \ }t\in (a+b-x,b].
\end{equation*}%
Therefore we can write%
\begin{eqnarray*}
&&\left\vert \frac{1}{b-a}\int_{a}^{b}f(t)dt-\frac{1}{2}\left[ f(x)+f(a+b-x)%
\right] \right. \\
&&\left. +\frac{1}{2}\left( x-\frac{3a+b}{4}\right) \left[ f^{\prime
}(x)-f^{\prime }(a+b-x)\right] \right\vert \\
&\leq &\frac{1}{2(b-a)}\left\{ \int_{a}^{x}\left( t-a\right) ^{2}\left[
\left( \frac{t-a}{x-a}\right) ^{s}\left\vert f^{\prime \prime
}(x)\right\vert +\left( \frac{x-t}{x-a}\right) ^{s}\left\vert f^{\prime
\prime }(a)\right\vert \right] dt\right. \\
&&\left. +\int_{x}^{a+b-x}\left( t-\frac{a+b}{2}\right) ^{2}\left[ \left( 
\frac{t-x}{a+b-2x}\right) ^{s}\left\vert f^{\prime \prime
}(a+b-x)\right\vert +\left( \frac{a+b-x-t}{a+b-2x}\right) ^{s}\left\vert
f^{\prime \prime }(x)\right\vert \right] dt\right. \\
&&\left. +\int_{a+b-x}^{b}\left( t-b\right) ^{2}\left[ \left( \frac{t-a-b+x}{%
x-a}\right) ^{s}\left\vert f^{\prime \prime }(b)\right\vert +\left( \frac{b-t%
}{x-a}\right) ^{s}\left\vert f^{\prime \prime }(a+b-x)\right\vert \right]
dt\right\} \\
&=&\frac{1}{2(b-a)}\left\{ \frac{1}{\left( 3+s\right) }\left( x-a\right)
^{3}\left\vert f^{\prime \prime }(x)\right\vert +\frac{2}{\left( s+1\right)
\left( s+2\right) \left( s+3\right) }\left( x-a\right) ^{3}\left\vert
f^{\prime \prime }(a)\right\vert \right. \\
&&\left. +\frac{s^{2}+s+2}{4\left( s+1\right) \left( s+2\right) \left(
s+3\right) }\left( a+b-2x\right) ^{3}\left\vert f^{\prime \prime
}(a+b-x)\right\vert +\frac{s^{2}+s+2}{4\left( s+1\right) \left( s+2\right)
\left( s+3\right) }\left( a+b-2x\right) ^{3}\left\vert f^{\prime \prime
}(x)\right\vert \right. \\
&&\left. +\frac{2}{\left( s+1\right) \left( s+2\right) \left( s+3\right) }%
\left( x-a\right) ^{3}\left\vert f^{\prime \prime }(b)\right\vert +\frac{1}{%
\left( 3+s\right) }\left( x-a\right) ^{3}\left\vert f^{\prime \prime
}(a+b-x)\right\vert \right\} \\
&=&\frac{\left( x-a\right) ^{3}}{\left( s+1\right) \left( s+2\right) \left(
s+3\right) (b-a)}\left[ \left\vert f^{\prime \prime }(a)\right\vert
+\left\vert f^{\prime \prime }(b)\right\vert \right] \\
&&+\frac{4\left( s^{2}+3s+2\right) \left( x-a\right) ^{3}+\left(
s^{2}+s+2\right) \left( a+b-2x\right) ^{3}}{8\left( s+1\right) \left(
s+2\right) \left( s+3\right) (b-a)}\left[ \left\vert f^{\prime \prime
}(x)\right\vert +\left\vert f^{\prime \prime }(a+b-x)\right\vert \right] ,
\end{eqnarray*}%
which is the desired result.
\end{proof}

\begin{corollary}
\label{co 2.1} Let $f$ as in Theorem \ref{teo 2.1}. Additionally, if $%
f^{\prime }(x)=f^{\prime }(a+b-x)$, we have%
\begin{eqnarray*}
&&\left\vert \frac{1}{b-a}\int_{a}^{b}f(t)dt-\frac{1}{2}\left[ f(x)+f(a+b-x)%
\right] \right\vert \\
&\leq &\frac{\left( x-a\right) ^{3}}{\left( s+1\right) \left( s+2\right)
\left( s+3\right) (b-a)}\left[ \left\vert f^{\prime \prime }(a)\right\vert
+\left\vert f^{\prime \prime }(b)\right\vert \right] \\
&&+\frac{4\left( s^{2}+3s+2\right) \left( x-a\right) ^{3}+\left(
s^{2}+s+2\right) \left( a+b-2x\right) ^{3}}{8\left( s+1\right) \left(
s+2\right) \left( s+3\right) (b-a)}\left[ \left\vert f^{\prime \prime
}(x)\right\vert +\left\vert f^{\prime \prime }(a+b-x)\right\vert \right] .
\end{eqnarray*}
\end{corollary}

\begin{corollary}
\label{co 2.2} In Corollary \ref{co 2.1}, if $f$ is symmetric function, $%
f(a+b-x)=f(x),$ for all $x\in \left[ a,\frac{a+b}{2}\right] $ we have 
\begin{eqnarray*}
&&\left\vert \frac{1}{b-a}\int_{a}^{b}f(t)dt-f(x)\right\vert \\
&\leq &\frac{\left( x-a\right) ^{3}}{\left( s+1\right) \left( s+2\right)
\left( s+3\right) (b-a)}\left[ \left\vert f^{\prime \prime }(a)\right\vert
+\left\vert f^{\prime \prime }(b)\right\vert \right] \\
&&+\frac{4\left( s^{2}+3s+2\right) \left( x-a\right) ^{3}+\left(
s^{2}+s+2\right) \left( a+b-2x\right) ^{3}}{8\left( s+1\right) \left(
s+2\right) \left( s+3\right) (b-a)}\left[ \left\vert f^{\prime \prime
}(x)\right\vert +\left\vert f^{\prime \prime }(a+b-x)\right\vert \right] ,
\end{eqnarray*}%
which is an Ostrowski type inequality.
\end{corollary}

\begin{corollary}
\label{co 2.3} In Theorem \ref{teo 2.1}, if we choose
\end{corollary}

(1) $x=\frac{a+b}{2},$ we have 
\begin{eqnarray*}
&&\left\vert \frac{1}{b-a}\int_{a}^{b}f(t)dt-f\left( \frac{a+b}{2}\right)
\right\vert \\
&\leq &\frac{\left( b-a\right) ^{2}}{8\left( s+1\right) \left( s+2\right)
\left( s+3\right) }\left[ \left\vert f^{\prime \prime }(a)\right\vert
+\left( s^{2}+3s+2\right) \left\vert f^{\prime \prime }\left( \frac{a+b}{2}%
\right) \right\vert +\left\vert f^{\prime \prime }(b)\right\vert \right] .
\end{eqnarray*}

(2) $x=\frac{3a+b}{4},$ we have 
\begin{eqnarray*}
&&\left\vert \frac{1}{b-a}\int_{a}^{b}f(t)dt-\frac{1}{2}\left[ f\left( \frac{%
3a+b}{4}\right) +f\left( \frac{a+3b}{4}\right) \right] \right\vert \\
&\leq &\frac{\left( b-a\right) ^{2}}{128\left( s+1\right) \left( s+2\right)
\left( s+3\right) }\left[ 2\left\vert f^{\prime \prime }(a)\right\vert
+(3s^{2}+5s+6)\left[ \left\vert f^{\prime \prime }\left( \frac{3a+b}{4}%
\right) \right\vert +\left\vert f^{\prime \prime }\left( \frac{a+3b}{4}%
\right) \right\vert \right] +2\left\vert f^{\prime \prime }(b)\right\vert %
\right] .
\end{eqnarray*}

\begin{theorem}
\label{teo 2.2} Let $f:[a,b]\rightarrow 
\mathbb{R}
$ be a function such that $f^{\prime }$ is absolutely continuous on $[a,b],$ 
$f^{\prime \prime }\in L_{1}[a,b].$ If $\left\vert f^{\prime \prime
}\right\vert ^{q}$ is $s-$convex on $[a,b],$ for all $x\in \left[ a,\frac{a+b%
}{2}\right] $ and $q>1,$ then we have the following inequality:%
\begin{eqnarray*}
&&\left\vert \frac{1}{b-a}\int_{a}^{b}f(t)dt-\frac{1}{2}\left[ f(x)+f(a+b-x)%
\right] \right. \\
&&\left. +\frac{1}{2}\left( x-\frac{3a+b}{4}\right) \left[ f^{\prime
}(x)-f^{\prime }(a+b-x)\right] \right\vert \\
&\leq &\frac{1}{2\left( b-a\right) \left( 2p+1\right) ^{\frac{1}{p}}\left(
s+1\right) ^{\frac{1}{q}}}\left[ \left( x-a\right) ^{3}\left( \left\vert
f^{\prime \prime }(a)\right\vert ^{q}+\left\vert f^{\prime \prime
}(x)\right\vert ^{q}\right) ^{\frac{1}{q}}\right. \\
&&\left. +\frac{\left( a+b-2x\right) ^{3}}{4}\left( \left\vert f^{\prime
\prime }(x)\right\vert ^{q}+\left\vert f^{\prime \prime }(a+b-x)\right\vert
^{q}\right) ^{\frac{1}{q}}\right. \\
&&\left. +\left( x-a\right) ^{3}\left( \left\vert f^{\prime \prime
}(a+b-x)\right\vert ^{q}+\left\vert f^{\prime \prime }(b)\right\vert
^{q}\right) ^{\frac{1}{q}}\right] ,
\end{eqnarray*}%
where $\frac{1}{p}+\frac{1}{q}=1.$
\end{theorem}

\begin{proof}
Using Lemma \ref{lem 1.1}, H\"{o}lder inequality and $s-$convexity of $%
\left\vert f^{\prime \prime }\right\vert ^{q},$ we have%
\begin{eqnarray*}
&&\left\vert \frac{1}{b-a}\int_{a}^{b}f(t)dt-\frac{1}{2}\left[ f(x)+f(a+b-x)%
\right] \right. \\
&&\left. +\frac{1}{2}\left( x-\frac{3a+b}{4}\right) \left[ f^{\prime
}(x)-f^{\prime }(a+b-x)\right] \right\vert \\
&\leq &\frac{1}{2\left( b-a\right) }\left\{ \left( \int_{a}^{x}\left(
t-a\right) ^{2p}dt\right) ^{\frac{1}{p}}\left( \int_{a}^{x}\left\vert
f^{\prime \prime }(t)\right\vert ^{q}dt\right) ^{\frac{1}{q}}\right. \\
&&\left. +\left( \int_{x}^{a+b-x}\left( t-\frac{a+b}{2}\right)
^{2p}dt\right) ^{\frac{1}{p}}\left( \int_{x}^{a+b-x}\left\vert f^{\prime
\prime }(t)\right\vert ^{q}dt\right) ^{\frac{1}{q}}\right. \\
&&\left. +\left( \int_{a+b-x}^{b}\left( t-b\right) ^{2p}dt\right) ^{\frac{1}{%
p}}\left( \int_{a+b-x}^{b}\left\vert f^{\prime \prime }(t)\right\vert
^{q}dt\right) ^{\frac{1}{q}}\right\} \\
&\leq &\frac{1}{2\left( b-a\right) }\left\{ \left( \int_{a}^{x}\left(
t-a\right) ^{2p}dt\right) ^{\frac{1}{p}}\left( \int_{a}^{x}\left[ \left( 
\frac{t-a}{x-a}\right) ^{s}\left\vert f^{\prime \prime }(x)\right\vert
^{q}+\left( \frac{x-t}{x-a}\right) ^{s}\left\vert f^{\prime \prime
}(a)\right\vert ^{q}\right] dt\right) ^{\frac{1}{q}}\right. \\
&&\left. +\left( \int_{x}^{a+b-x}\left( t-\frac{a+b}{2}\right)
^{2p}dt\right) ^{\frac{1}{p}}\left( \int_{x}^{a+b-x}\left[ \left( \frac{t-x}{%
a+b-2x}\right) ^{s}\left\vert f^{\prime \prime }(a+b-x)\right\vert
^{q}+\left( \frac{a+b-x-t}{a+b-2x}\right) ^{s}\left\vert f^{\prime \prime
}(x)\right\vert ^{q}\right] dt\right) ^{\frac{1}{q}}\right. \\
&&\left. +\left( \int_{a+b-x}^{b}\left( t-b\right) ^{2p}dt\right) ^{\frac{1}{%
p}}\left( \int_{a+b-x}^{b}\left[ \left( \frac{t-a-b+x}{x-a}\right)
^{s}\left\vert f^{\prime \prime }(b)\right\vert ^{q}+\left( \frac{b-t}{x-a}%
\right) ^{s}\left\vert f^{\prime \prime }(a+b-x)\right\vert ^{q}\right]
dt\right) ^{\frac{1}{q}}\right\} \\
&=&\frac{1}{2\left( b-a\right) }\left\{ \left( \frac{\left( x-a\right)
^{2p+1}}{\left( 2p+1\right) }\right) ^{\frac{1}{p}}\left( \frac{x-a}{s+1}%
\right) ^{\frac{1}{q}}\left( \left\vert f^{\prime \prime }(a)\right\vert
^{q}+\left\vert f^{\prime \prime }(x)\right\vert ^{q}\right) ^{\frac{1}{q}%
}\right. \\
&&\left. +\left( \frac{2}{2p+1}\left( \frac{a+b}{2}-x\right) ^{2p+1}\right)
^{\frac{1}{p}}\left( \frac{a+b-2x}{s+1}\right) ^{\frac{1}{q}}\left(
\left\vert f^{\prime \prime }(x)\right\vert ^{q}+\left\vert f^{\prime \prime
}(a+b-x)\right\vert ^{q}\right) ^{\frac{1}{q}}\right. \\
&&\left. +\left( \frac{\left( x-a\right) ^{2p+1}}{\left( 2p+1\right) }%
\right) ^{\frac{1}{p}}\left( \frac{x-a}{s+1}\right) ^{\frac{1}{q}}\left(
\left\vert f^{\prime \prime }(a+b-x)\right\vert ^{q}+\left\vert f^{\prime
\prime }(b)\right\vert ^{q}\right) ^{\frac{1}{q}}\right\} .
\end{eqnarray*}%
When we arrange the statements above, we obtain the desired result.
\end{proof}

\begin{corollary}
\label{co 2.5} Let $f$ as in Theorem \ref{teo 2.2}. Additionally, if $%
f^{\prime }(x)=f^{\prime }(a+b-x)$, we have%
\begin{eqnarray*}
&&\left\vert \frac{1}{b-a}\int_{a}^{b}f(t)dt-\frac{1}{2}\left[ f(x)+f(a+b-x)%
\right] \right\vert \\
&\leq &\frac{1}{2\left( b-a\right) \left( 2p+1\right) ^{\frac{1}{p}}\left(
s+1\right) ^{\frac{1}{q}}}\left[ \left( x-a\right) ^{3}\left( \left\vert
f^{\prime \prime }(a)\right\vert ^{q}+\left\vert f^{\prime \prime
}(x)\right\vert ^{q}\right) ^{\frac{1}{q}}\right. \\
&&\left. +\frac{\left( a+b-2x\right) ^{3}}{4}\left( \left\vert f^{\prime
\prime }(x)\right\vert ^{q}+\left\vert f^{\prime \prime }(a+b-x)\right\vert
^{q}\right) ^{\frac{1}{q}}\right. \\
&&\left. +\left( x-a\right) ^{3}\left( \left\vert f^{\prime \prime
}(a+b-x)\right\vert ^{q}+\left\vert f^{\prime \prime }(b)\right\vert
^{q}\right) ^{\frac{1}{q}}\right] ,
\end{eqnarray*}%
for all $x\in \left[ a,\frac{a+b}{2}\right] .$
\end{corollary}

\begin{corollary}
\label{co 2.6} In Corollary \ref{co 2.5}, if $f$ is symmetric function, $%
f(a+b-x)=f(x),$ we have 
\begin{eqnarray*}
&&\left\vert \frac{1}{b-a}\int_{a}^{b}f(t)dt-f(x)\right\vert \\
&\leq &\frac{1}{2\left( b-a\right) \left( 2p+1\right) ^{\frac{1}{p}}\left(
s+1\right) ^{\frac{1}{q}}}\left[ \left( x-a\right) ^{3}\left( \left\vert
f^{\prime \prime }(a)\right\vert ^{q}+\left\vert f^{\prime \prime
}(x)\right\vert ^{q}\right) ^{\frac{1}{q}}\right. \\
&&\left. +\frac{\left( a+b-2x\right) ^{3}}{4}\left( \left\vert f^{\prime
\prime }(x)\right\vert ^{q}+\left\vert f^{\prime \prime }(a+b-x)\right\vert
^{q}\right) ^{\frac{1}{q}}\right. \\
&&\left. +\left( x-a\right) ^{3}\left( \left\vert f^{\prime \prime
}(a+b-x)\right\vert ^{q}+\left\vert f^{\prime \prime }(b)\right\vert
^{q}\right) ^{\frac{1}{q}}\right] ,
\end{eqnarray*}%
for all $x\in \left[ a,\frac{a+b}{2}\right] .$
\end{corollary}

\begin{corollary}
\label{co 2.7} In Theorem \ref{teo 2.2}, if we choose
\end{corollary}

(1) $x=\frac{a+b}{2},$ we have%
\begin{eqnarray*}
&&\left\vert \frac{1}{b-a}\int_{a}^{b}f(t)dt-f\left( \frac{a+b}{2}\right)
\right\vert \\
&\leq &\frac{\left( b-a\right) ^{2}}{\left( s+1\right) ^{\frac{1}{q}}\left(
2p+1\right) ^{\frac{1}{p}}16}\left[ \left( \left\vert f^{\prime \prime
}(a)\right\vert ^{q}+\left\vert f^{\prime \prime }\left( \frac{a+b}{2}%
\right) \right\vert ^{q}\right) ^{\frac{1}{q}}\right. \\
&&\left. \left( \left\vert f^{\prime \prime }\left( \frac{a+b}{2}\right)
\right\vert ^{q}+\left\vert f^{\prime \prime }(b)\right\vert ^{q}\right) ^{%
\frac{1}{q}}\right] .
\end{eqnarray*}

(2) $x=\frac{3a+b}{4},$ we have%
\begin{eqnarray*}
&&\left\vert \frac{1}{b-a}\int_{a}^{b}f(t)dt-\frac{1}{2}\left[ f\left( \frac{%
3a+b}{4}\right) +f\left( \frac{a+3b}{4}\right) \right] \right\vert \\
&\leq &\frac{\left( b-a\right) ^{2}}{128\left( 2p+1\right) ^{\frac{1}{p}%
}\left( s+1\right) ^{\frac{1}{q}}}\left\{ \left( \left\vert f^{\prime \prime
}(a)\right\vert ^{q}+\left\vert f^{\prime \prime }\left( \frac{3a+b}{4}%
\right) \right\vert ^{q}\right) ^{\frac{1}{q}}\right. \\
&&\left. +2\left( \left\vert f^{\prime \prime }\left( \frac{3a+b}{4}\right)
\right\vert ^{q}+\left\vert f^{\prime \prime }\left( \frac{a+3b}{4}\right)
\right\vert ^{q}\right) ^{\frac{1}{q}}\right. \\
&&\left. +\left( \left\vert f^{\prime \prime }\left( \frac{a+3b}{4}\right)
\right\vert ^{q}+\left\vert f^{\prime \prime }(b)\right\vert ^{q}\right) ^{%
\frac{1}{q}}\right\} .
\end{eqnarray*}

\begin{corollary}
\label{co 2.8} In Corollary \ref{co 2.5}, if we choose $x=a$ we have%
\begin{equation*}
\left\vert \frac{1}{b-a}\int_{a}^{b}f(t)dt-\frac{f(a)+f(b)}{2}\right\vert
\leq \frac{\left( b-a\right) ^{2}}{8\left( 2p+1\right) ^{\frac{1}{p}}\left(
s+1\right) ^{\frac{1}{q}}}\left[ \left\vert f^{\prime \prime }(a)\right\vert
^{q}+\left\vert f^{\prime \prime }(b)\right\vert ^{q}\right] ^{\frac{1}{q}}.
\end{equation*}
\end{corollary}

\begin{remark}
\label{rem 2.1} Using the well-known power-mean integral inequality one may
get inequalities for functions whose second derivatives absolute value are
convex. The details are omitted.
\end{remark}

We obtain the following result for $s-$concave functions.

\begin{theorem}
\label{teo 2.3} Let $f:[a,b]\rightarrow 
\mathbb{R}
$ be a function such that $f^{\prime }$ is absolutely continuous on $[a,b],$ 
$f^{\prime \prime }\in L_{1}[a,b].$ If $\left\vert f^{\prime \prime
}\right\vert ^{q}$ is $s-$concave on $[a,b],$ for all $x\in \left[ a,\frac{%
a+b}{2}\right] $ and $q>1,$ then we have the following inequality:%
\begin{eqnarray*}
&&\left\vert \frac{1}{b-a}\int_{a}^{b}f(t)dt-\frac{1}{2}\left[ f(x)+f(a+b-x)%
\right] \right. \\
&&\left. +\frac{1}{2}\left( x-\frac{3a+b}{4}\right) \left[ f^{\prime
}(x)-f^{\prime }(a+b-x)\right] \right\vert \\
&\leq &\frac{2^{\frac{s-1}{q}}}{2\left( b-a\right) \left( 2p+1\right) ^{%
\frac{1}{p}}}\left[ \left( x-a\right) ^{3}\left\vert f^{\prime \prime
}\left( \frac{x+a}{2}\right) \right\vert \right. \\
&&\left. +\frac{\left( a+b-2x\right) ^{3}}{4}\left\vert f^{\prime \prime
}\left( \frac{a+b}{2}\right) \right\vert +\left( x-a\right) ^{3}\left\vert
f^{\prime \prime }\left( \frac{a+2b-x}{2}\right) \right\vert \right]
\end{eqnarray*}%
where $\frac{1}{p}+\frac{1}{q}=1.$
\end{theorem}

\begin{proof}
From Lemma \ref{lem 1.1} and using H\"{o}lder inequality, we have%
\begin{eqnarray*}
&&\left\vert \frac{1}{b-a}\int_{a}^{b}f(t)dt-\frac{1}{2}\left[ f(x)+f(a+b-x)%
\right] \right. \\
&&\left. +\frac{1}{2}\left( x-\frac{3a+b}{4}\right) \left[ f^{\prime
}(x)-f^{\prime }(a+b-x)\right] \right\vert \\
&\leq &\frac{1}{2\left( b-a\right) }\left\{ \left( \int_{a}^{x}\left(
t-a\right) ^{2p}dt\right) ^{\frac{1}{p}}\left( \int_{a}^{x}\left\vert
f^{\prime \prime }(t)\right\vert ^{q}dt\right) ^{\frac{1}{q}}\right. \\
&&\left. +\left( \int_{x}^{a+b-x}\left( t-\frac{a+b}{2}\right)
^{2p}dt\right) ^{\frac{1}{p}}\left( \int_{x}^{a+b-x}\left\vert f^{\prime
\prime }(t)\right\vert ^{q}dt\right) ^{\frac{1}{q}}\right. \\
&&\left. +\left( \int_{a+b-x}^{b}\left( t-b\right) ^{2p}dt\right) ^{\frac{1}{%
p}}\left( \int_{a+b-x}^{b}\left\vert f^{\prime \prime }(t)\right\vert
^{q}dt\right) ^{\frac{1}{q}}\right\} \\
&=&\frac{1}{2\left( b-a\right) }\left\{ \left( \int_{a}^{x}\left( t-a\right)
^{2p}dt\right) ^{\frac{1}{p}}\left( \int_{a}^{x}\left\vert f^{\prime \prime
}\left( \frac{t-a}{x-a}x+\frac{x-t}{x-a}a\right) \right\vert ^{q}dt\right) ^{%
\frac{1}{q}}\right. \\
&&\left. +\left( \int_{x}^{a+b-x}\left( t-\frac{a+b}{2}\right)
^{2p}dt\right) ^{\frac{1}{p}}\left( \int_{x}^{a+b-x}\left\vert f^{\prime
\prime }\left( \frac{t-x}{a+b-2x}\left( a+b-x\right) +\frac{a+b-x-t}{a+b-2x}%
x\right) \right\vert ^{q}dt\right) ^{\frac{1}{q}}\right. \\
&&\left. +\left( \int_{a+b-x}^{b}\left( t-b\right) ^{2p}dt\right) ^{\frac{1}{%
p}}\left( \int_{a+b-x}^{b}\left\vert f^{\prime \prime }\left( \frac{t-a-b+x}{%
x-a}b+\frac{b-t}{x-a}\left( a+b-x\right) \right) \right\vert ^{q}dt\right) ^{%
\frac{1}{q}}\right\} .
\end{eqnarray*}%
Since $\left\vert f^{\prime \prime }\right\vert ^{q}$ is $s-$concave, from (%
\ref{1.1}), we have 
\begin{equation*}
\int_{a}^{x}\left\vert f^{\prime \prime }\left( \frac{t-a}{x-a}x+\frac{x-t}{%
x-a}a\right) \right\vert ^{q}dt\leq 2^{s-1}\left( x-a\right) \left\vert
f^{\prime \prime }\left( \frac{a+x}{2}\right) \right\vert ^{q},
\end{equation*}%
\begin{equation*}
\int_{x}^{a+b-x}\left\vert f^{\prime \prime }\left( \frac{t-x}{a+b-2x}\left(
a+b-x\right) +\frac{a+b-x-t}{a+b-2x}x\right) \right\vert ^{q}dt\leq
2^{s-1}\left( a+b-2x\right) \left\vert f^{\prime \prime }\left( \frac{a+b}{2}%
\right) \right\vert ^{q}
\end{equation*}%
and%
\begin{equation*}
\int_{a+b-x}^{b}\left\vert f^{\prime \prime }\left( \frac{t-a-b+x}{x-a}b+%
\frac{b-t}{x-a}\left( a+b-x\right) \right) \right\vert ^{q}dt\leq
2^{s-1}\left( x-a\right) \left\vert f^{\prime \prime }\left( \frac{a+2b-x}{2}%
\right) \right\vert ^{q}.
\end{equation*}%
Combining all above inequalities, we obtain%
\begin{eqnarray*}
&&\left\vert \frac{1}{b-a}\int_{a}^{b}f(t)dt-\frac{1}{2}\left[ f(x)+f(a+b-x)%
\right] \right. \\
&&\left. +\frac{1}{2}\left( x-\frac{3a+b}{4}\right) \left[ f^{\prime
}(x)-f^{\prime }(a+b-x)\right] \right\vert \\
&\leq &\frac{1}{2\left( b-a\right) }\left\{ \left( \frac{\left( x-a\right)
^{2p+1}}{\left( 2p+1\right) }\right) ^{\frac{1}{p}}2^{\frac{s-1}{q}}\left(
x-a\right) ^{\frac{1}{q}}\left\vert f^{\prime \prime }\left( \frac{x+a}{2}%
\right) \right\vert \right. \\
&&\left. +\left( \frac{2}{2p+1}\left( \frac{a+b}{2}-x\right) ^{2p+1}\right)
^{\frac{1}{p}}2^{\frac{s-1}{q}}\left( a+b-2x\right) ^{\frac{1}{q}}\left\vert
f^{\prime \prime }\left( \frac{a+b}{2}\right) \right\vert \right. \\
&&\left. +\left( \frac{\left( x-a\right) ^{2p+1}}{\left( 2p+1\right) }%
\right) ^{\frac{1}{p}}2^{\frac{s-1}{q}}\left( x-a\right) ^{\frac{1}{q}%
}\left\vert f^{\prime \prime }\left( \frac{a+2b-x}{2}\right) \right\vert
\right\} \\
&\leq &\frac{2^{\frac{s-1}{q}}}{2\left( b-a\right) \left( 2p+1\right) ^{%
\frac{1}{p}}}\left[ \left( x-a\right) ^{3}\left\vert f^{\prime \prime
}\left( \frac{x+a}{2}\right) \right\vert \right. \\
&&\left. +\frac{\left( a+b-2x\right) ^{3}}{4}\left\vert f^{\prime \prime
}\left( \frac{a+b}{2}\right) \right\vert +\left( x-a\right) ^{3}\left\vert
f^{\prime \prime }\left( \frac{a+2b-x}{2}\right) \right\vert \right]
\end{eqnarray*}%
for all $x\in \left[ a,\frac{a+b}{2}\right] $ and $\frac{1}{p}+\frac{1}{q}%
=1. $
\end{proof}

\begin{corollary}
\label{co 2.9} Let $f$ as in Theorem \ref{teo 2.3}. Additionally, if $%
f^{\prime }(x)=f^{\prime }(a+b-x)$, we have%
\begin{eqnarray*}
&&\left\vert \frac{1}{b-a}\int_{a}^{b}f(t)dt-\frac{1}{2}\left[ f(x)+f(a+b-x)%
\right] \right\vert \\
&\leq &\frac{2^{\frac{s-1}{q}}}{2\left( b-a\right) \left( 2p+1\right) ^{%
\frac{1}{p}}}\left[ \left( x-a\right) ^{3}\left( \left\vert f^{\prime \prime
}\left( \frac{x+a}{2}\right) \right\vert +\left\vert f^{\prime \prime
}\left( \frac{a+2b-x}{2}\right) \right\vert \right) \right. \\
&&\left. +\frac{\left( a+b-2x\right) ^{3}}{4}\left\vert f^{\prime \prime
}\left( \frac{a+b}{2}\right) \right\vert \right] ,
\end{eqnarray*}%
for all $x\in \left[ a,\frac{a+b}{2}\right] .$
\end{corollary}

\begin{corollary}
\label{co 2.10} In Corollary \ref{co 2.9}, if $f$ is symmetric function, $%
f(a+b-x)=f(x),$ we have 
\begin{eqnarray*}
&&\left\vert \frac{1}{b-a}\int_{a}^{b}f(t)dt-f(x)\right\vert  \\
&\leq &\frac{2^{\frac{s-1}{q}}}{2\left( b-a\right) \left( 2p+1\right) ^{%
\frac{1}{p}}}\left[ \left( x-a\right) ^{3}\left( \left\vert f^{\prime \prime
}\left( \frac{x+a}{2}\right) \right\vert +\left\vert f^{\prime \prime
}\left( \frac{a+2b-x}{2}\right) \right\vert \right) \right.  \\
&&\left. +\frac{\left( a+b-2x\right) ^{3}}{4}\left\vert f^{\prime \prime
}\left( \frac{a+b}{2}\right) \right\vert \right] ,
\end{eqnarray*}%
for all $x\in \left[ a,\frac{a+b}{2}\right] .$
\end{corollary}

\begin{corollary}
\label{c0 2.11} In Theorem \ref{teo 2.3}, if we choose $x=\frac{3a+b}{4}$ we
have%
\begin{eqnarray*}
&&\left\vert \frac{1}{b-a}\int_{a}^{b}f(t)dt-\frac{1}{2}\left[ f\left( \frac{%
3a+b}{4}\right) +f\left( \frac{a+3b}{4}\right) \right] \right\vert \\
&\leq &\frac{2^{\frac{s-1}{q}}\left( b-a\right) ^{2}}{128\left( 2p+1\right)
^{\frac{1}{p}}}\left[ \left\vert f^{\prime \prime }\left( \frac{7a+b}{8}%
\right) \right\vert +2\left\vert f^{\prime \prime }\left( \frac{a+b}{2}%
\right) \right\vert +\left\vert f^{\prime \prime }\left( \frac{a+7b}{8}%
\right) \right\vert \right] .
\end{eqnarray*}
\end{corollary}

\begin{remark}
\label{rem 2.2} IIn Theorem \ref{teo 2.3}, if we choose $x=\frac{a+b}{2},$
we have the first inequality in Corollary \ref{co 1.1}.
\end{remark}

\section{applications for special means}

We consider the means for nonnegative real numbers $\alpha <\beta $ as
follows:

(1)\textbf{\ }The arithmetic mean:%
\begin{equation*}
A\left( \alpha ,\beta \right) =\frac{\alpha +\beta }{2},\text{ \ \ \ \ }%
\alpha ,\beta \in 
\mathbb{R}
^{+}.
\end{equation*}

(2) The generalized logarithmic mean:%
\begin{equation*}
L_{n}\left( \alpha ,\beta \right) =\left[ \frac{\beta ^{n+1}-\alpha ^{n+1}}{%
\left( \beta -\alpha \right) \left( n+1\right) }\right] ^{\frac{1}{n}},\text{
\ \ \ \ }\alpha ,\beta \in 
\mathbb{R}
^{+},\text{ }\alpha \neq \beta ,\text{ }n\in 
\mathbb{Z}
\backslash \left\{ -1,0\right\} .
\end{equation*}

\begin{proposition}
\label{prop 3.1} Let $0<a<b$ and $0<s<1.$ Then we have%
\begin{equation*}
\left\vert L_{s}^{s}\left( a,b\right) -A^{s}\left( a,b\right) \right\vert
\leq \frac{\left( b-a\right) ^{2}s\left( s-1\right) }{8\left( s+1\right)
\left( s+2\right) \left( s+3\right) }\left[ a^{s-2}+\left( s^{2}+3s+2\right)
A^{s-2}\left( a,b\right) +b^{s-2}\right]
\end{equation*}%
and%
\begin{eqnarray*}
&&\left\vert L_{s}^{s}\left( a,b\right) -A\left( \left( \frac{3a+b}{4}%
\right) ^{s},\left( \frac{a+3b}{4}\right) ^{s}\right) \right\vert \\
&\leq &\frac{\left( b-a\right) ^{2}s\left( s-1\right) }{8\left( s+1\right)
\left( s+2\right) \left( s+3\right) }\left[ 2\left( a^{s-2}+b^{s-2}\right)
+\left( 3s^{2}+5s+6\right) \left( \left( \frac{3a+b}{4}\right) ^{s-2}+\left( 
\frac{a+3b}{4}\right) ^{s-2}\right) \right] .
\end{eqnarray*}
\end{proposition}

\begin{proof}
The assertion follows from Corollary \ref{co 2.3} applied to the $s-$convex
mapping $f:[0,1]\rightarrow \lbrack 0,1],$ $f(x)=x^{s}.$
\end{proof}

\begin{proposition}
\label{prop 3.2} Let $0<a<b$ and $0<s<1.$ Then for all $q>1,$ we have%
\begin{eqnarray*}
\left\vert L_{s}^{s}\left( a,b\right) -A^{s}\left( a,b\right) \right\vert
&\leq &\frac{\left( b-a\right) ^{2}s\left( s-1\right) }{16\left( s+1\right)
^{\frac{1}{q}}\left( 2p+1\right) ^{\frac{1}{p}}}\left[ \left( a^{\left(
s-2\right) q}+A^{\left( s-2\right) q}\left( a,b\right) \right) ^{\frac{1}{q}%
}\right. \\
&&\left. +\left( \left( A^{\left( s-2\right) q}\left( a,b\right) +b^{\left(
s-2\right) q}\right) ^{\frac{1}{q}}\right) \right]
\end{eqnarray*}%
where $\frac{1}{p}+\frac{1}{q}=1.$
\end{proposition}

\begin{proof}
The assertion follows from first inequality in Corollary \ref{co 2.7}
applied to the $s-$convex mapping $f:[0,1]\rightarrow \lbrack 0,1],$ $%
f(x)=x^{s}.$
\end{proof}

\begin{proposition}
\label{prop 3.3} Let $0<a<b$ and $0<s<1.$ Then for all $q>1,$ we have%
\begin{equation*}
\left\vert L_{s}^{s}\left( a,b\right) -A\left( a^{s},b^{s}\right)
\right\vert \leq \frac{\left( b-a\right) ^{2}s\left( s-1\right) }{8\left(
s+1\right) ^{\frac{1}{q}}\left( 2p+1\right) ^{\frac{1}{p}}}\left[ a^{\left(
s-2\right) q}+b^{\left( s-2\right) q}\right] ^{\frac{1}{q}},
\end{equation*}%
where $\frac{1}{p}+\frac{1}{q}=1.$
\end{proposition}

\begin{proof}
The assertion follows from first inequality in Corollary \ref{co 2.8}
applied to the $s-$convex mapping $f:[0,1]\rightarrow \lbrack 0,1],$ $%
f(x)=x^{s}.$
\end{proof}

\end{document}